\documentclass[12pt]{amsart}
\usepackage[margin=3cm]{geometry}
\usepackage{amsmath,amssymb,amsbsy}
\usepackage[english]{babel}
\usepackage{todonotes}
\usepackage{comment}
\usepackage{url}
\usepackage[colorlinks,linkcolor=blue,anchorcolor=blue,citecolor=blue,backref=page]{hyperref}
\usepackage{enumitem}

\newcommand{\pdiv}{\mid\!\mid}

\numberwithin{equation}{section}

\newtheorem{theorem}{Theorem}[section]
\newtheorem{lemma}[theorem]{Lemma}
\newtheorem{proposition}[theorem]{Proposition}

\newcommand{\RR}{\operatorname{R}}
\newcommand{\M}{\operatorname{M}}
\newcommand{\E}{\operatorname{S}}

\newcommand{\N}{\mathbb N}
\newcommand{\Z}{\mathbb Z}
\newcommand{\R}{\mathbb R}
\newcommand{\seq}{\subseteq}

\title{On products of sets of natural density one}

\author{Sandro Bettin}
\address{DIMA - Dipartimento di Matematica, Via Dodecaneso, 35, 16146 Genova, Italy}
\email{sandro.bettin@unige.it}

\author{Matteo Bordignon}
\address{Department of Mathematics, KTH, SE-100 44 Stockholm, Sweden}
\email{bordig@kth.se}

\author{Alessandro Fazzari}
\address{D\'epartement de math\'ematiques et de statistique, Universit\'e de Montr\'eal. CP 6128, succ. Centre-ville. Montreal, QC H3C 3J7, Canada}
\email{alessandro.fazzari@umontreal.ca}


\date{\today}

\begin{document}
\begin{abstract} 
In a previous work, Bettin, Koukoulopoulos, and Sanna prove that if two sets of natural numbers $A$ and $B$ have natural density $1$, then their product set $A \cdot B := \{ab : a \in A, b \in B\}$ also has natural density $1$. They also provide an effective rate and pose the question of determining the optimal rate. We make progress on this question by constructing a set $A$ of density 1 such that $A\cdot A$ has a ``large'' complement.
\end{abstract}

\maketitle

\section{Introduction}
The study of product sets $A \cdot B := \{ab : a \in A, b \in B\}$ of two sets of natural numbers $A$ and $B$ has long been of interest in mathematics. For finite sets, the classic multiplication table problem, posed by Erd\"{o}s~\cite{Erdos1, Erdos2}, seeks bounds on the cardinality of the $n\times n$ multiplication table. This problem was fully resolved by Ford~\cite{Ford}, building on earlier work by Tenenbaum~\cite{Tenenbaum1987}. A multidimensional variation was later studied by Koukoulopoulos~\cite{Koukoulopoulos}.
For more general finite sets, the cardinality problem has been investigated by Cilleruelo, Ramana, and Ramaré~\cite{CRR}, as well as by Mastrostefano~\cite{Mastrostefano} and Sanna~\cite{Sanna}.

The analogous problem for infinite sets of natural numbers was considered by Hegyvári, Hennecart, and Pach~\cite{HHP}. In this context, the role of cardinality is played by the natural density $d(A):= \lim_{x\rightarrow \infty } \frac{\# (A \cap [1,x])}{x}$ of a set $A$, if the limit exists.  
Hegyvári, Hennecart, and Pach asked whether, given two sets $A,B$ with density $1$, the product set $A \cdot B$ also has density $1$.

In~\cite{BKS}, Bettin, Koukoulopoulos, and Sanna answered this question in the affirmative. In other words, defining  
\begin{align}\label{rfrx}
\RR_x(A) := 1 - \frac{\# (A \cap [1,x])}{x}
\end{align}
for any $A \subseteq \mathbb{N}$ and $x \geq 1$, they proved that if $\RR_x(A),\RR_x(B) \to 0$ as $x \to \infty$, then also $\RR_x(A \cdot B) \to 0$. 
In the same paper, it was also remarked that one could obtain an explicit rate of convergence for $\RR_x(A \cdot A)$ in terms of the rate of $\RR_x(A)$. More specifically, their proof (cf.~\cite[Remark, p.1411]{BKS}) gives that if  
$$\RR_x(A)\ll (\log x)^{-a} \quad \text{ for some } a\in(0,1),$$
then 
\begin{align*}
%\label{eq:upperb}
    \RR_x(A\cdot A)\ll (\log x)^{-\frac{a^2}{1+a}+o(1)}.
\end{align*}
Equivalently, letting
\begin{align*}
\psi(a)&:=\sup \Big\{b\in\R_{>0}\ \Big|\ \RR_{x}(A\cdot A)\ll  (\log x)^{-b} \ \forall A\seq\N\text{ s.t. }\RR_x(A)\ll(\log x)^{-a}\Big\}\\
&\hphantom{:}=\inf_{A\seq \mathbb N \atop \RR_x(A)\ll (\log x)^{-a}\hspace{-1.em}}\big\{b\in\R_{>0}\ \big|\  \RR_x(A\cdot A)=\Omega( (\log x)^{-b} )\big\}
\end{align*}
for $a>0$, the result of~\cite{BKS} establishes that $\psi(a) \geq a^2/(1+a)$ for $a\in(0,1)$. In this note, we aim to make progress on determining the function $\psi(a)$ by providing an upper bound.
%for $a>0$, one has that $\psi(a)\geq \frac{a^2}{1+a}$ for $a\in(0,1)$, with the authors of~\cite{BKS} also posing the question of determining the function $\psi(a)$.
%the result of~\cite{BKS} can be rephrased as $\psi(a)\geq \frac{a^2}{1+a}$ for $a\in(0,1)$ and the authors of~\cite{BKS} then mention that it would be interesting to determine the function $\psi(a)$. 
%%, with the authors of~\cite{BKS} also proposing the question of determining the function $\psi(a)$. 
%In this note we seek to make some progress on this problem by providing an upper bound for $\psi$. 
It is easy to see that $\psi(a)\leq a$ for all $a\in(0,1)$. Indeed, let us denote by $P$ a subset of the primes with relative asymptotic density $a\in(0,1)$, i.e. $\symbol{35}(P\cap[1,x])\sim ax/\log x$. Then, letting 
$$A_P:=\{n\in\N\mid \exists p\in P\text{ s.t. $p|n$}\}\cup\{1\}$$ 
we have $A_P\cdot A_P=A_P$ and, by the Fundamental Lemma of Sieve Theory~\cite[Theorem 18.11]{DimitrisBook},
$$\RR_x(A_P)=\RR_x(A_P\cdot A_P)=\frac1x\#\{n\in[2,x]\mid (p,n)=1\, \forall p\in P\}
=(\log x)^{-a+o(1)}.$$
We improve upon this ``trivial'' bound for sufficiently small values of $a$. More specifically, we show the following.

\begin{theorem}\label{mathe}
For $a\in(0,1/4)$, $B\in  [0,\phi^{-1}(4a)]$, let 
\begin{align}\label{aewa}
W(a,B):=\frac{(\phi(B)-\sqrt{\phi(B)^2-4a \phi(B)})^2}{4\phi(B)}+a\frac{\phi(2B)}{\phi(B)}
\end{align}
where $\phi:[0,1]\to[0,1]$ is defined by 
\begin{align}\label{dfp}
\phi(x):=\begin{cases}
x\log x-x+1, & x\in(0,1),\\
1 &x=0.
\end{cases}\end{align}
Then, defining 
\begin{align}\label{aewa2}
K(a)&:= \min_{\substack{ B\in [0,\phi^{-1}(4a)]}}W(a,B),\qquad a\in(0,1/4), 
\end{align}
we have
\begin{align}\label{24Jan.7}
\psi(a)\leq K(a).
\end{align}
Moreover, one has
\begin{align}
&\label{thh.01}K(a)<a && \text{ for } a\in(0,0.11717],\\
&\label{thh.02}K(a)\le \frac{2a^2}{1-\log 2}+o(a^2) && \text{ as }a\to0^+. 
\end{align}
\end{theorem}

\begin{figure}[ht]
\includegraphics[width=10cm]{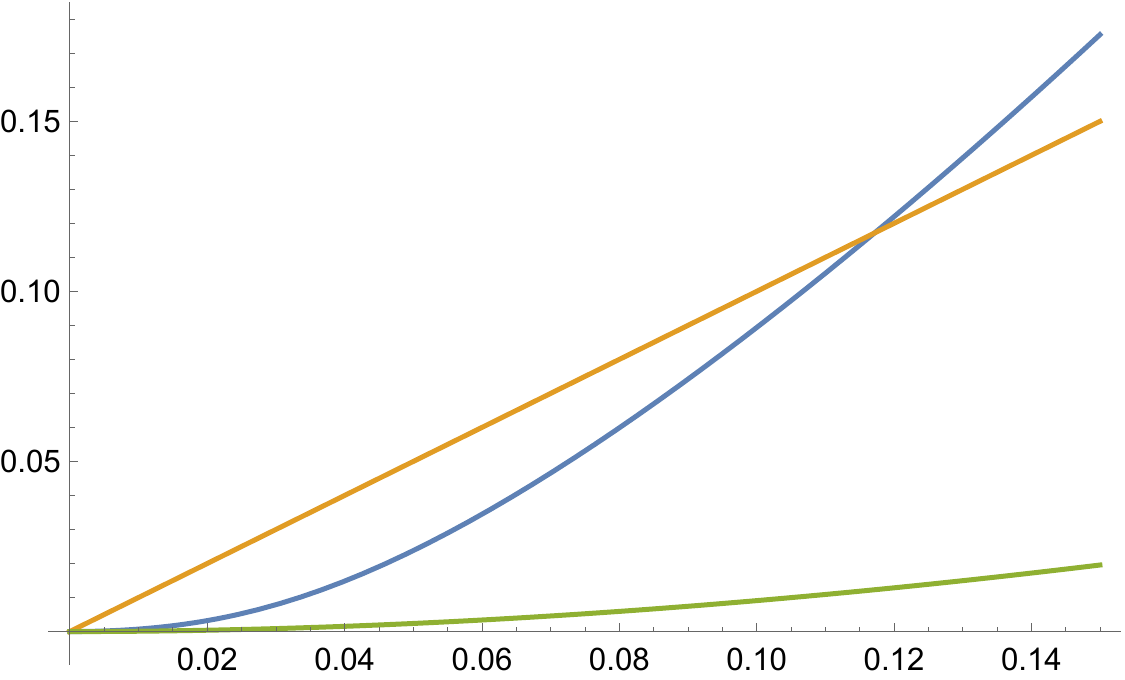}
\caption{The functions $\frac{a^2}{1+a}$ (green), $K(a)$ (blue) and $a$ (orange) for $0\leq a\leq0.15$. The function $\psi$ lies between the green curve and the minimum between the blue and the orange curves. }
\end{figure}

We conclude the Introduction by noting that, together with the bound of~\cite{BKS},~\eqref{thh.02} implies that $\psi(a)$ decays quadratically as $a\to0^+$. More precisely, 
$$1\leq\liminf_{a\to0^+}\frac{\psi(a)}{a^2}\leq \limsup_{a\to0^+}\frac{\psi(a)}{a^2}\leq \frac{2}{1-\log2}=6.51778\dots.$$

\subsection{Sketch of the argument}

The main idea used in~\cite{BKS} is as follows: any integer $n$ can be factorized as  
\[
n = n_{\text{smooth}} \cdot n_{\text{rough}},
\]  
where $n_{\text{smooth}}$ and $n_{\text{rough}}$ are the products of its ``small'' and ``large'' prime factors, respectively, with respect to a suitably chosen (small) cutoff. If $n \notin A \cdot A$, then at least one of these factors must be missing from $A$, meaning either $n_{\text{smooth}} \notin A$ or $n_{\text{rough}} \notin A$.  
If the product set $A \cdot A$ does not have density $1$, then $A$ must lack its expected proportion of either smooth or rough numbers. Consequently, $A$ itself cannot have density $1$.  

In the argument above, for any $n\notin A\cdot A$ one infers information about $A$ from a single factorization of $ n $, where $ n $ is written as a product of two integers. These integer, in addition to not both belonging to $ A $, also satisfy the extra condition of being respectively small and smooth, and large and rough. To construct a set $ A $ such that $ \mathbb{N} \setminus (A \cdot A) $ is large, we aim to define $ A $ in a way that naturally forces a typical integer $ m $ to have all its factorizations in $ A \cdot A $ constrained by this extra condition.

To achieve this, we define $A$ as the set of integers that do not have ``too few'' large prime divisors, i.e., we consider  
\[
A = \big\{ m : \Omega^*(m) > B \M(m) \big\},  
\]
where  
\[
\Omega^*(m) = \sum_{\substack{p^{\nu_p} \pdiv m \\ \exp(\delta \log\log m) < p \leq m}} \nu_p,\qquad \M(m) = (1-\delta)\log\log m.
\]  
Notice that $\M(m)$ is the expected average value of $\Omega^*(m)$. The parameters $\delta, B \in (0,1)$ are then chosen so that  
$
\RR_x(A) = (\log x)^{-a}.
$  

For any factorization $m = n_1 n_2$ with $n_1, n_2 \in A$ and $n_1 \leq n_2$, we have two possibilities:
\begin{itemize}
\item[(a)] If $n_1$ and $n_2$ are of comparable size, then  
$  \Omega^*(m) \gtrsim B ( \M(n_1) + \M(n_2) ) \approx 2B \M(m)$
\item[(b)] If $n_1$ is much smaller than $n_2$, then $m$ has at least $B \M(n_1)$ prime divisors smaller than $n_1$, and thus $m$ does not have too few prime divisors of such (small) size.  
\end{itemize}  

Both conditions on $m$ are stricter (in terms of asymptotic cardinality) than the condition in the definition of $A$. Despite $m$ needing to satisfy only one of these conditions (and in fact, condition (b) applies across all possible ranges of $n_1$), we obtain that  $\RR_x(A \cdot A)$ is larger than $ \RR_x(A)$.  
We then optimize the choice of $\delta$ and $B$ to maximize $\RR_x(A \cdot A)$.  

When making this argument rigorous, we need to make an additional refinement. Specifically, we modify $\Omega^*(m)$ to ``discretize'' the interval $(\exp(\delta \log\log m), m]$ in its definition. See~\eqref{sdl} for the precise definition of the set $A$. This adjustment is needed when handling all the possible range constraints in case (b).

\subsection{Notations}
Throughout the paper, we will employ the following standard notations. Given integers $a,b,$ and $m$, we write $a| b$ if $a$ divides $b$, and $a^m \!\pdiv\! b$ if $a^m$ divides $b$ exactly, i.e. $a^m|b$ and $a^{m+1}\nmid b$.
We also employ Landau’s notation $f= O(g)$ and Vinogradov’s notation $f \ll g$, both meaning that $|f|\leq C|g|$ for some constant $C > 0$. If the constant $C$ depends on some parameter $y$, we write $f=O_y(g)$ or $f\ll_y g$. The notation $f= o(g)$ as $x\to a$ means that $\lim_{x\to a} f(x)/g(x) = 0$. Finally, we write $f=\Omega_y(|g|)$ as $x\to a$ if there exist a constant $c=c(y)>0$ and a sequence $x_n\to a$ such that $|f(x_n)|\geq c|g(x_n)|$.

\subsection*{Acknowledgments} 
The authors wish to thank Michel Balazard, Andrew Granville, Tony Haddad, and Dimitris Koukoulopoulos for inspiring conversations.
Part of this work was completed while the authors were in residence at the Institut Mittag-Leffler in Djursholm Sweden during the spring semester of 2024, and is supported by the Swedish Research Council under grant no. 2021-06594. 
S.B. is partially supported by PRIN 2022 “The arithmetic of motives and L-functions”, by the Curiosity Driven grant “Value distribution of quantum modular forms” of the University of Genoa, funded by the European Union -- NextGenerationEU, and by the MIUR Excellence Department Project awarded to Dipartimento di Matematica, Università di Genova, CUP D33C23001110001. 
M.B. was partially supported by the Swedish Research Council (2020-04036).
A.F. is supported by the Fonds de recherche du Qu\'ebec - Nature et technologies, Projet de recherche en \'equipe 300951, and thanks P\"ar Kurlberg for the invitation to KTH, where part of this work was completed.
Finally, S.B. and A.F. are members of the INdAM group GNAMPA.

\section{Lemmata}

For any set of primes $S$, we denote
$$\E(x):= \sum_{\substack{p\leq x \\ p\in S}}\frac{1}{p},\qquad \Omega(n;S) := \sum_{\substack{p^{\nu_p} \pdiv n \\ p\in S}} \nu_p. $$

All the preliminary results stated in this section are manifestations of the Poissonian nature of the arithmetic function $\Omega(n;S)$.
The first one is a standard upper bound for the probability that a random integer $n$ has a limited number of prime divisors in an interval.

\begin{lemma}\label{NewLemma01}
Let $\phi$ be as in~\eqref{dfp}. Then, uniformly for $e<U<V\leq \log\log x$ and $B\in[0,1)$, we have 
$$\frac{1}{x}\sum_{\substack{n\leq x \atop \Omega(n;(U,V])\leq B\log\frac{\log V}{\log U}}}1 \ll \bigg(\frac{\log V}{\log U}\bigg)^{-\phi(B)} . $$ 
\end{lemma}

\begin{proof}
Hal\'asz~\cite{Halasz2} proved sharp bounds for integers $n\leq x$ with $\Omega(n;(U,V])=k$. To obtain the claimed result it suffices to sum over $k\leq B\log\frac{\log V}{\log U}$.
%It suffices to sum the bound given by Lemma \ref{NewLemma00} over $k=0,\dots,B(V-U)$; the claim then follows from Stirling's formula. Indeed, since $S(x)=V-U+O(1)$ when $S=(e^{e^U},e^{e^V}]$ and $e^{e^V}\leq x$, then
%\begin{equation}\begin{split}\notag
%\frac{1}{x}\sum_{\substack{n\leq x \\ \Omega_{U,V}(n)\leq B(V-U)}}1 
%= \sum_{k=0}^{B(V-U)}\frac{1}{x}\sum_{\substack{n\leq x \\ \Omega_{U,V}(n)= k}}1 
%&\ll \sum_{k=0}^{B(V-U)} \frac{(V-U)^k}{k!}e^{-(V-U)}\\
%&%= \frac{\Gamma(1+B(V-U),V-U)}{\Gamma(1+B(V-U))} 
%\ll \frac{e^{-(V-U)}e^{B(V-U)}}{B^{B(V-U)}} \ll e^{-\phi(B)(V-U)},
%\end{split}\end{equation}
%as per the Poissonian distribution law.
\end{proof}

 %but with the conditions on the divisors imposed on multiple disjoint intervals. 
Building on works of Hal\'asz~\cite{Halasz1, Halasz2}, S\'ark\"ozy \cite{Sarkozy} (see also~\cite{Balazard}, Theorem A and subsequent paragraphs on page 391) obtained a lower bound for the number of integers with $\Omega(n;S)=k$. We need a version of this (with $\Omega(n;S)\leq k$) where there are multiple conditions on the number of prime divisors in disjoint sets. This is obtained in Tenenbaum~\cite{Tenenbaum2017}, but only when $k$ is not too small. See also~\cite{Mangerel, Ford1,Tudesq} for some related results.  We provide a short proof of the precise result that we need by making simple modifications to~\cite{Halasz2}, being very brief in the steps that are essentially identical to Hal\'asz' work.
 
\begin{lemma}\label{NewLemma4}
Let $m\in\mathbb N$, $\underline k = (k_1,\dots,k_m)\in\mathbb N^m$, $\varepsilon>0$ and $x\geq1$. Let $S_1,\dots,S_m$ disjoint sets of primes. Then, for $A_\varepsilon$ large enough we have
$$ N(\underline k,x):=\sum_{\substack{n\leq x \\ \Omega(n;S_j) \leq k_j\; \; \forall j=1,\dots, m}} 1  
\gg  x\prod_{j=1}^{m}\frac{\E_j(x)^{k_j-1}}{(k_j-1)!}e^{-\E_j(x)} $$
uniformly in $x,\underline k$ satisfying $1 \leq k_j \leq (2-\varepsilon) \E_j(x)$ and $\E_j(x)\geq A_\varepsilon$ for all $j=1,\dots,m$.
\end{lemma}

\begin{proof}
Since $\log n\geq\sum_{p| n}\log p$, for all $u\leq 2x$ we have
\begin{align*}
 N(\underline{k},2x) \geq \sum_{\substack{n\leq u \\ \Omega(n,S_j)\leq k_j\; \forall j}} \frac{\log n }{\log 2x}
\geq \sum_{p\leq u} \frac{\log p}{\log 2x} \sum_{\substack{h\leq u/p \\ \Omega(ph;S_j)\leq k_j\; \forall j}} 1\geq \sum_{p\leq u} \frac{\log p}{\log 2x} N(\underline{k}-\underline{1},u/p),
\end{align*}
with $\underline 1=(1,\dots,1)$. Dividing by $2x$ and integrating over $u\leq 2x$ we then have
\begin{equation}\begin{split}\label{fewx}
N(\underline{k},2x) &\geq  \int_1^{2x} \sum_{p\leq u}  N(\underline{k}-\underline 1,u/p)\frac{\log p}{\log 2x}\,\frac{du}{2x}%\\&=\sum_{p\leq 2x}\frac{\log p}{\log x}\int_p^{2x}   N(\underline{k}-\underline{1},u/p)\,\frac{du}{2x}=\sum_{p\leq 2x}p\frac{\log p}{\log x}\int_1^{2x/p}   N(\underline{k}-\underline{1},u)\,\frac{du}{2x}\\
=\int_1^{x} \sum_{p\leq 2x/u}p\frac{\log p}{\log 2x} \, N(\underline{k}-\underline{1},u)\,\frac{du}{2x}\\
&\gg \frac1{\log x}\int_1^{x} \frac{x}{u^2} \, N(\underline{k}-\underline{1},u)\,du\geq \frac{x}{\log x}\int_1^{x} \frac{N(\underline{k}-\underline{1},u)}{u^{1+\sigma}}  du
\end{split}\end{equation}
for any $\sigma\geq1$.
By integration by parts, we have 
\begin{align}\label{fewx2}
\sum_{\substack{n\leq  x \\ \Omega(n;S_j) = k_j-1\; \forall j}} \frac{1}{n^{\sigma}}\leq\sum_{\substack{n\leq  x \\ \Omega(n;S_j) \leq k_j-1\; \forall j}} \frac{1}{n^{\sigma}}=\frac{N(\underline{k}-\underline{1},x)}{x^{\sigma}}
+\sigma\int_1^{x} \frac{ N(\underline{k}-\underline{1},u)}{u^{1+\sigma}}\,du
\end{align}
and thus, since $N(\underline{k}-\underline{1},x)\leq N(\underline{k},2x)$, for $x$ sufficiently large~\eqref{fewx}-\eqref{fewx2} yield
\begin{align}\label{dsac}
N(\underline{k},2x)&\gg \frac{x}{\log x}\sum_{\substack{n\leq  x \\ \Omega(n;S_j) = k_j-1 \; \forall j}} \frac{1}{n^{\sigma}}.
\end{align}
Now, let $(r_1,\dots,r_m)\in \mathbb R_{>0}^m$.  Assuming $\sigma>1$, by Cauchy's theorem we have
\begin{align}\label{cint}
\sum_{\substack{n=1 \\ \Omega(n;S_j) = k_j-1\; \; \forall j}}^\infty \frac{1}{n^{\sigma}}=\frac1{(2\pi i)^m}\int_{|z_1|=r_1}\cdots\int_{|z_m|=r_m}\frac{F(\underline{z},\sigma)}{z_1^{k_1}\cdots z_m^{k_m}}\,dz_1\cdots dz_{m}
\end{align}
where for $\underline z=(z_1,\dots,z_m)\in\mathbb C^m$
\begin{align*} 
F(\underline{z},\sigma) := \sum_{n=1}^{\infty} \frac{z_1^{\Omega(n;S_1)}\cdots z_m^{\Omega(n;S_m)}}{n^\sigma} 
= \exp\bigg(\sum_{j=1}^m \sum_{p\in S_j} \sum_{\ell=1}^{\infty}\frac{z_j^\ell}{\ell p^{\ell \sigma}}+\sum_{p\notin \cup_j\!S_j}\sum_{\ell=1}^{\infty}\frac{1}{\ell p^{\ell \sigma}} \bigg)
\end{align*}
where the second expression is obtained by expanding $F$ in its Euler's product. We assume $|z_j|=r_j\leq 2-\varepsilon\; \forall j$ and pick $\sigma=\sigma_v=1+\frac1{\log v}$ with $2\leq v\leq x$. We compare $F(\underline{z},\sigma_v)$ with $F(\underline{r},\sigma_v)$; a simple computation yields 
\begin{align}\label{evr0}
F(\underline{z},\sigma_v)&= F(\underline{r},\sigma_v) \exp\bigg(\sum_{j=1}^m (z_j - r_j) \E_j(v) + O_{\varepsilon}\Big( \sum_{j=1}^m|z_j-r_j|\Big)\bigg)
\end{align}
and\begin{align} \label{evr}
F(\underline{r},\sigma_v) = \zeta(\sigma) \exp\bigg( \sum_{j=1}^m \sum_{p\in S_j\atop p\leq v}\sum_{\ell=1}^{\infty} \frac{r_j-1}{p^{\sigma}}+O(1)\bigg)=e^{\sum_{j=1}^m(r_j-1) \E_j(v)+ O(1)} \log v.
\end{align}
We insert~\eqref{evr0} into~\eqref{cint}. For the main term we evaluate the integrals, and we estimate the contribution of the error choosing $r_j:=k_j/\E_j(v)\leq 2-\varepsilon$ and using the inequality $|e^{z \E_j(v)}|\leq e^{r\E_j(v)}e^{-\theta^2\E_j(v)}$ for $z=re^{i\theta}$, $\theta\in[-\pi,\pi]$. Using also~\eqref{evr} we obtain
\begin{align}
\sum_{\substack{n=1 \\ \Omega(n;S_j) = k_j-1\; \forall j}}^\infty \frac{1}{n^{\sigma_v}}  &=F(\underline{r},\sigma_v)\prod_{j=1}^m\int_{|z_j|=r_j}\frac{\exp\big( (z_j-r_j) \E_j(v)\big)}{z_j^{k_j}}(1+O_{\varepsilon}(|z_j-r_j|))\,dz_j\notag\\
&=(\log v)e^{O(1)}\prod_{j=1}^m \frac{\E_j(v)^{k_j-1}e^{- \E_j(v)} }{(k_j-1)!}(1+O_{\varepsilon}(\E_j(v)^{-1/2})).\label{ftg}
\end{align}
Finally, for $C>2$ we let $y=x^{1/C}$. We have $0\leq \E_j(x)-\E_j(y)\leq\sum_{y<p\leq x}1/p=O(\log C)$ and so $\E_j(x)^{k_j-1}e^{- \E_j(x)} = \E_j(y)^{k_j-1}e^{- \E_j(y)}e^{O(\log C)} $. Thus,
\begin{align*}
\sum_{\substack{n> x \\ \Omega(n;S_j) = k_j-1\; \forall j}} \frac{1}{n^{\sigma_y}}  \leq x^{\sigma_x-\sigma_y}\sum_{\substack{n=1 \\ \Omega(n;S_j) = k_j-1\; \forall j}}^{\infty} \frac{1}{n^{\sigma_{x}}}\sim Ce^{-C+O(\log C)}\log y\prod_{j=1}^m \frac{\E_j(y)^{k_j-1}e^{- \E_j(y)} }{(k_j-1)!}.
\end{align*}
We fix $C$  large enough so that $Ce^{-C+O(\log C)}$ is sufficiently small and deduce by~\eqref{ftg}
\begin{align*}
\sum_{\substack{n\leq  x \\ \Omega(n;S_j) = k_j-1\; \; \forall j}} \frac{1}{n^{\sigma_y}} \gg \sum_{\substack{n=1 \\ \Omega(n;S_j) = k_j-1\; \; \forall j}}^\infty \frac{1}{n^{\sigma_y}}\gg  (\log x)\prod_{j=1}^m \frac{\E_j(y)^{k_j-1}e^{- \E_j(y)} }{(k_j-1)!}
\end{align*}
for $A_\varepsilon$ large enough. The claimed bound then follows by~\eqref{dsac}.
\end{proof}

\section{Proof of Theorem~\ref{mathe}}

Let $y>1$ and $0< B < 1$ be two real parameters. Notice that $1/y$ plays the role of the parameter $\delta$ in the introduction.
 Let us also introduce the notations 
$$E_k := e^{e^{y^{ k}}} \quad \text{ and } \quad D_k := y^{ k}-y^{k-1}, \quad \text{for } k\in\mathbb Z . $$ 
Also, let $c=c(y):=\frac{1}{y}(1-\frac{1}{y})$, so that in particular $D_{k-1}=y^{ k}c$. For the sake of brevity, we denote 
$$ \Omega_{k}(n) := \Omega \big(n;(E_{k-1},E_k]\big)
%= \sum_{\substack{p^{\nu_p}||n \\ e^{e^{e^{\alpha(k-1)}}} < p \leq e^{e^{e^{\alpha k}}}} }\nu_p
= \sum_{\substack{ p^{\nu_p}||n \\ E_{k-1} < p \leq E_k}} \nu_p .$$
%
%$\Omega_k(n)$ the counting function of prime factors of $n$ (with multiplicity) between $E_{k-1}$ and $E_k$; namely
%$$ \Omega_{k}(n) %= \Omega_{y^{k-1},y^{k}} (n)
%%= \sum_{\substack{p^{\nu_p}||n \\ e^{e^{e^{\alpha(k-1)}}} < p \leq e^{e^{e^{\alpha k}}}} }\nu_p
%:= \sum_{\substack{ p^{\nu_p}||n \\ E_{k-1} < p \leq E_k}} \nu_p .$$
Finally, we introduce the following set:
\begin{align} \label{sdl}
A:=\{ n\in (E_0,\infty) \mid \Omega_{k_n}(n) > \max\{1,BD_{k_n}\} \} \quad \text{ with }\quad k_n := \max\{k\in\Z_{\geq0}\mid E_k< n\} . 
\end{align}

\subsection{The density of \texorpdfstring{$A$}{A} and \texorpdfstring{$A \cdot A$}{A · A}}
As a first step towards Theorem \ref{mathe}, we establish an upper bound for the asymptotic density of the complement of $A$. We recall that $\RR_x$ is as defined in~\eqref{rfrx}.
\begin{proposition}\label{sdlf}
In the above notations, for $x\geq2$ we have
$$ \RR_x(A) 
%= \frac{1}{x} \sum_{\substack{ n\leq x \\ n\notin A }}1  
\ll_y \frac{1}{(\log x)^{\phi(B)c}} . $$
\end{proposition}

\begin{proof}
We write $x$ as $x=E_{M-z}$ with $M\in \mathbb N$ and $z\in[0,1)$, so that $\log\log x=y^{M-z}$. We have
\begin{align}\label{wwt}
x \RR_x(A) 
%= \sum_{\substack{ n\leq x \\ n\notin A }}1 + O(1)
=   \sum_{\substack{ 1\leq n\leq E_{M-2} \\ n\notin A }}1 + \sum_{\substack{E_{M-2}<  n\leq E_{M-1} \\ n\notin A }}1 +   \sum_{\substack{ E_{M-1}<n\leq E_{M-z} \\ n\notin A }}1.
\end{align}
%We will bound each of the three sums on the right-hand side individually.
First we note that
\begin{equation*}%\label{24Jan.1} 
E_{M-2} = \exp(e^{y^{M-2}}) = \exp( (\log x)^{y^{z-2}} ) = x^{o(1)},
\end{equation*}
since $z-2<0$ and $y>1$. 
In particular, the first of the three sums in~\eqref{wwt} is negligible.
Moreover, if $E_{M-2}< n \leq E_{M-1}$ then $k_n = M-2$, so by definition of $A$ we have
\begin{equation}\begin{split}\label{23Jan.0}
\sum_{\substack{E_{M-2}<n\leq E_{M-1}\\ n\not\in A }} 1
=\sum_{\substack{E_{M-2}<n\leq E_{M-1}\\ \Omega_{M-2}(n)< BD_{M-2}}} 1 
\ll \frac{E_{M-1}}{e^{\phi(B)D_{M-2}}}= \frac{E_{M-1}}{e^{\phi(B)y^{M-1}c}}
%=\frac{E_{M-1}}{e^{\phi(B)D_{M-2}}}= \frac{E_{M-1}}{(\log N)^{\phi(B)y^{z-1}c}}
%\leq \frac{N}{e^{\phi(B)y^{M-z-1} c}}\\
%&=\frac{N}{(\log N)^{\phi(B) c/y}}=o\bigg(\frac{N}{(\log N)^{\phi(B) c}}\bigg)
\end{split}\end{equation} 
by Lemma~\ref{NewLemma01}. 
We note that for any $z\in[0,1)$ we have 
\begin{equation}\label{23Jan.1}
\frac{E_{M-1}}{e^{\phi(B)y^{M-1}c}}\ll_y \frac{x}{(\log x)^{\phi(B)c}} .
\end{equation} 
Indeed, for $0\leq z< 1-\frac{1}{\sqrt{\log x}}$ one has
\begin{equation}\notag
\frac{E_{M-1}}{e^{\phi(B)y^{M-1}c}} 
<  E_{M-1} 
%= \exp(e^{y^{M-1}}) 
= \exp\Big((\log x)^{y^{z-1}}\Big)
< \exp\Big((\log x)^{y^{-1/\sqrt{\log x}}}\Big)
\ll_A \frac{x}{(\log x)^A}
\end{equation} 
for all $A>0$, whereas in the range $1-\frac{1}{\sqrt{\log x}}<z<1$, we write $z=1-O\big(\frac{1}{\sqrt{\log x}}\big)$ and get
\begin{equation}\notag
\frac{E_{M-1}}{e^{\phi(B)y^{M-1}c}} 
\leq  \frac{x}{(\log x)^{\phi(B)c\,y^{z-1}}}
=  \frac{x}{(\log x)^{\phi(B)c+O_y(\frac{1}{\sqrt{\log x}})}}
\ll_y \frac{x}{(\log x)^{\phi(B)c}}  .
\end{equation} 
Then \eqref{23Jan.1} is proven and, together with~\eqref{23Jan.0}, yields that the second sum in~\eqref{wwt} is $O_y (x(\log x)^{-\phi(B)c} ).$
%\begin{equation}\begin{split}\notag
%\sum_{\substack{E_{M-2}<n\leq E_{M-1}\\ n\not\in A }} 1 \ll_y \frac{x}{(\log x)^{\phi(B)c}} .
%\end{split}\end{equation} 
Finally, we deal with the third sum in~\eqref{wwt}. For $n\in(E_{M-1},E_{M-z}]$  we have $k_n = M-1$. Hence, by applying Lemma~\ref{NewLemma01} we obtain
\begin{equation*}%\label{15apr.2}
\sum_{\substack{E_{M-1}< n \leq E_{M-z} \\ n\not\in A }} 1
\ll \frac{x}{e^{\phi(B)D_{M-1}}} 
%= \frac{N}{e^{\phi(B)y^Mc}}
= \frac{x}{(\log x)^{\phi(B) c \, y^z}} 
\ll \frac{x}{(\log x)^{\phi(B)c}}
\end{equation*}
and the proof is completed.
\end{proof}

We now prove an Omega result for the asymptotic density of the complement of $A\cdot A$. 

\begin{proposition}\label{sln}
In the above notations, we have
\begin{align}\label{csdc}
\RR_x(A\cdot A) 
%= \frac{1}{x} \sum_{\substack{ n\leq x \\ n\notin A \cdot A }}1  +O\bigg(\frac{1}{x}\bigg)
= \Omega_y\bigg(\frac{1}{(\log x)^{\phi(B)y^{-2}+\phi(\min\{1,2B\})c-o(1)}}\bigg).
\end{align}
\end{proposition}

\begin{proof}
Let $x=E_{M}-1$ for some $M\in\mathbb N$. We will show that~\eqref{csdc} holds with $\gg$ for such values of $x$.

Let $n\in (A\cdot A)\cap[1,x]$. Then, at least one of the following must hold:
\vspace{.5em}
\begin{itemize}
\item[a)]
$\displaystyle n\in \Big(\big(E_{M-1},E_{M}\big)\cap A\Big)\cdot \Big(\big(E_{M-1},E_{M}\big)\cap A\Big)$;
\vspace{.5em}
\item[b)] $n$ has a divisor in $(E_{M-r-1},E_{M-r}]\cap A$ for some $r\in\{1,\dots,M-1\}$.
\end{itemize}
\vspace{.5em}
%Indeed, if neither of the above conditions holds, then $n\in[1,E_{M-1})\cdot[1,E_{M-1})$. With the same argument as in \eqref{24Jan.1}, one sees that these values of $n$ contribute at most $x^{-1+o(1)}$.

By definition of $A$, the condition a) forces $\Omega_{M-1}(n)>2BD_{M-1}$, whereas case b) implies $\Omega_{M-r-1}(n)>B D_{M-r-1}$ for some $r\geq 1$. It follows that for $M$ large enough
\begin{align*}
(\mathbb N\cap[1,x])\setminus\big(A\cdot A\big)&\supseteq\bigg\{n\leq x\ \bigg|\ {\Omega_{M-1}(n)\leq 2B D_{M-1},\atop \Omega_{M-r-1}(n)\leq \max\{1,B D_{M-r-1}\}\  \forall r=1,\dots,M-2}\bigg\}.
\end{align*}
We fix $r_0\in[1,M-2]$ and let $J=(E_0,E_{M-r_0-1}]$. Clearly, $\Omega(n;J)=\sum_{r=r_0}^{M-2}\Omega_{M-r-1}(n)$. Thus, we have
\begin{align*}
(\mathbb N\cap[1,x])\setminus\big(A\cdot A\big)&\supseteq\bigg\{n\leq x\ \bigg|\ {\Omega_{M-1}(n)\leq \min\{1,2B\} D_{M-1},\ \ \Omega(n;J)\leq 1,\atop \Omega_{M-r-1}(n)\leq B D_{M-r-1}\  \forall r=1,\dots,r_0-1}\bigg\}.
\end{align*}
The set above is defined by conditions on the prime divisors of $n$ in disjoint intervals. Therefore, we can apply Lemma~\ref{NewLemma4} and obtain
\begin{equation}\begin{split}\notag
\RR_x(A\cdot A)
%\symbol{35}\Big( \mathbb N\cap[1,x] \setminus\big(A\cdot A\big)\Big)
\gg_{y,r_0} (\log x)^{o(1)} \times  \prod_{j=M-r_0}^{M-2}\frac{D_{j}^{BD_{j}-1}}{[BD_{j}]!}e^{-D_{j}} \times \frac{D_{M-1}^{\min\{1,2B\}D_{M-1}-1}}{[\min\{1,2B\}D_{M-1}]!} e^{-D_{M-1}}  \times e^{-y^{M-r_0-1}}.
\end{split}\end{equation}
\begin{comment}
\textcolor{gray}{
with
\begin{itemize}
\item $\displaystyle \Omega(n;T_j) = \Omega_j(n)$
\item $\displaystyle H(T_j) = D_{j}+O(1)$
\item $\displaystyle \ell_j = \begin{cases} [BD_{j}] & \text{if } j=1,\dots,M-2 \\ [\min(1,2B)D_{M-1}] & \text{if } j=M-1 \end{cases}$
\item $\displaystyle \frac{\ell_j}{H(T_j)} = \begin{cases} \frac{[BD_{j}]}{D_{j}+O(1)} & \text{if } j=1,\dots,M-2 \\ \frac{[\min(1,2B)D_{M-1}]}{D_{M-1}+O(1)} & \text{if } j=M-1 \end{cases} \approx 1$
\end{itemize}
\begin{equation}\begin{split}\notag
&\hspace{-2cm}= \prod_{j=1}^{M-2}\frac{(D_{j}+O(1))^{[BD_{j}]}}{[BD_{j}]!}e^{-D_{j}+O(1)} \times \frac{(D_{M-1}+O(1))^{[\min(1,2B)D_{M-1}]}}{[\min(1,2B)D_{M-1}]!} e^{-D_{M-1}+O(1)}  \times \exp\bigg(O\bigg( \sum_{j=1}^{M-2}1 \bigg)\bigg)  \\
&\hspace{-2cm}\gg (\log x)^{o(1)}  \prod_{j=1}^{M-2}\frac{D_{j}^{[BD_{j}]}}{[BD_{j}]!}e^{-D_{j}} \times \frac{D_{M-1}^{[\min(1,2B)D_{M-1}]}}{[\min(1,2B)D_{M-1}]!} e^{-D_{M-1}}  
\end{split}\end{equation}
since $\log \log x \approx y^M$ and then $M \approx \log\log\log x$, so $e^{O(M)}=e^{C\log\log\log x} = (\log\log x)^C = (\log x)^{o(1)}$.
}
\end{comment}
By Stirling's approximation formula, for fixed $b$, one has
%$$ \frac{D^{bD}}{(bD)!}e^{-D}
%\gg \frac{D^{bD}e^{-D}}{\sqrt{D}(bD)^{bD}} 
%= \frac{1}{\sqrt{D}} e^{-D(b\log b+b-1)}
%= \frac{1}{\sqrt{D}} e^{-D\phi(b)}. $$
$$  \frac{D^{bD-1}}{(bD)!}e^{-D}
\gg \frac{D^{bD}e^{-D}}{D\sqrt{D}(bD)^{bD}} 
%= \frac{1}{\sqrt{D}} e^{-D(b\log b+b-1)}
= \frac{1}{D^{3/2}} e^{-D\phi(b)} .$$
As a consequence,  we have 
\begin{equation}\begin{split}\notag
\RR_x(A\cdot A)
&\gg_{y,r_0} (\log x)^{o(1)} \times  \prod_{j=M-r_0}^{M-2}\frac{1}{D_{j}^{O(1)}} e^{-D_{j}\phi(B)} \times \frac{e^{-D_{M-1}\phi(\min\{1,2B\})-y^{M-r_0-1}}}{D_{M-1}^{O(1)}}  \\
&\gg_y (\log x)^{o(1)} \times  \exp\bigg( - \phi(B) \sum_{j=M-r_0}^{M-2} D_{j} -D_{M-1}\phi(\min\{1,2B\}) -y^{M-r_0-1}\bigg)
\end{split}\end{equation}
since $\prod_{j=M-r_0}^{M-1} D_{j}\ll (\log x)^{o(1)}$. Upon noting that the telescopic sum over $j$ above equals $y^{M-2}-y^{M-r_0-1}$, the above gives
\begin{equation}\begin{split}\notag
\RR_x(A\cdot A)
&\gg_{y,r_0} (\log x)^{o(1)} \times  \exp\Big( - y^M\big( \phi(B) y^{-2}+ c \, \phi(\min\{1,2B\}) +y^{-r_0-1}\big) \Big).
\end{split}\end{equation}
Since $y^M=\log\log x+O(1)$  we then obtain~\eqref{csdc} by letting $r_0\to\infty$ sufficiently slowly.
%[Qua, quando sapremo quale caso sia più conveniente, sarebbe meglio assumerlo direttamente e fare solo quel caso.]
%[Mettere dipendenza da $\alpha$ dove serve.]
%By  we have XXX
%\begin{align*}
%\#[1,N]\setminus\big(A\cdot A\big)&\gg_\alpha \frac{N}{(\log N)^{\phi(B)e^{-2\alpha}+\phi(\min(1,2B))c+o(1)}},
%\end{align*}
%as claimed.
\end{proof}

%In order to prove Theorem \ref{mathe} we now optimize on the choice of the parameters in our contruction. 

\subsection{Proof of Equation \texorpdfstring{\eqref{24Jan.7}}{}}
To establish Theorem \ref{mathe}, we now optimize the choice of parameters involved.
Let us consider $a$ to be in the image of the function $(0,1)\times(1,\infty) \ni (B,y)\mapsto \phi(B)\frac{1}{y} (1-\frac{1}{y})$, i.e. $a\in(0,\frac{1}{4})$. Then, by Proposition~\ref{sdlf} and~\ref{sln} we have
\begin{align}\label{24Jan.6}
\psi(a)\leq F(a),
\end{align}
where 
\begin{align*}
F(a)&= \inf_{\substack{y>1 \\ B\in(0,1)}}\Big\{ \phi(B)y^{-2}+\phi(\min\{1,2B\})y^{-1 }(1-y^{-1}) \ \Big|\  a=\phi(B) y^{- 1}(1-y^{-1}) \Big\}\\
%&= \min_{\substack{ B,t\in [0,1]}}\Big\{ \phi(B)t^2+\phi(\min(1,2B))t(1-t) \ \Big|\  a=\phi(B)t(1-t) \Big\}\\
&= \min_{ B,t\in [0,1]}\Big\{ \phi(B)t^2+\frac{\phi(\min\{1,2B\})}{\phi(B)}a \ \Big|\  a=\phi(B)t(1-t) \Big\}.\\
\end{align*}
The equation $a=\phi(B)t(1-t)$ has the unique solutions 
\begin{equation}\label{29Jan.1}
t'_{a,B}=\frac{\phi(B)+\sqrt{\phi(B)^2-4a\phi(B)}}{2\phi(B)} \quad
\text{ and } \quad t_{a,B}=\frac{\phi(B)-\sqrt{\phi(B)^2-4a\phi(B)}}{2\phi(B)}
\end{equation}
in $[0,1]$, and induces the condition $\phi(B)\geq4a$. By monotonicity, one immediately sees that the solution $t_{a,B}$ makes the above minimum smaller (and in fact the solution $t'_{a,B}$ yields the trivial lower bound $F(a)\geq a$).

Now, we have $t_{a,B}^2\phi(B)=G_a(\phi(B))$ with $G_a(x)=\frac{(x-\sqrt{x^2-4a x})^2}{4x}$. Note that $\phi(B)$ is strictly decreasing for $B\in[0,1]$ and $G_a$ is a strictly decreasing function. As a consequence, $t_{a,B}^2\phi(B)$ is strictly increasing for $B$ in its domain, $[0,\phi^{-1}(4a))$. Therefore, since $\frac{\phi(\min\{1,2B\})}{\phi(B)}a=0$ for $B\geq1/2$, the minimum must be attained for $B\leq1/2$. Hence, 
\begin{equation}\begin{split}\label{24Jan.5}
F(a)&= \min_{\substack{ B\in[0,\min\{\frac{1}{2},\phi^{-1}(4a)\}]}}\Big\{ \phi(B)t_{a,B}^2+a\frac{\phi(2B)}{\phi(B)}\Big\}\\
&= \min_{\substack{ B\in [0,\phi^{-1}(4a)]}}\Big\{ \phi(B)t_{a,B}^2+a\frac{\phi(2B)}{\phi(B)} \Big\}
=\min_{\substack{ B\in [0,\phi^{-1}(4a)]}}W(a,B)=K(a)
\end{split}\end{equation}
since $\frac{\phi(2B)}{\phi(B)}$ is also (strictly) increasing on $B\geq1/2$, and where $W$ and $K$ are defined in~\eqref{aewa} and~\eqref{aewa2}. Putting together \eqref{24Jan.6} and \eqref{24Jan.5}, one finally has \eqref{24Jan.7}.

%The function $\frac{d}{dB}\big(\phi(B)t_{a,B}^2+a\frac{\phi(2B)}{\phi(B)}\big)$ is convex in $a$ (it's useful to write it as a function of $P(B)$ first and then multiply it by $P'(B)$). Also it goes to $\infty$ as $a\to \phi(B)/4$ and is zero at $a=0$ . This means that for each $B$ one has another value of $a$ (beside $a=0$) for which $\frac{d}{dB}\big(\phi(B)t_{a,B}^2+a\frac{\phi(2B)}{\phi(B)}\big)=0$ [for this one needs to verify that the derivative at $a=0$ is negative]
%\todo[inline]{It is probably possible to show that given each $a$ the minimum is unique using some semi-numeric calculations, but it seems very annoying and then even more annoying to write.}

\subsection{Proof of Equation \texorpdfstring{\eqref{thh.01}}{}}
Let $a\in(0,1/4)$. If $\phi^{-1}(4a)>\frac{1}{2}$, i.e. $a<\frac{\phi(1/2)}{4}=\frac{1-\log 2}8=0.0383\dots$,  then we have 
$$K(a)\leq W (a,\tfrac12)=\phi(\tfrac12)t_{a,1/2}^2<\phi(\phi^{-1}(4a))t_{a,\phi^{-1}(4a)}^2=4a/4=a
$$
since $t_{a,B}^2\phi(B)$ is strictly increasing in $B$. Thus, we can assume $a\geq 0.038$.

Now let $\beta=5.3071678 $ and let $B_a$ be such that $\phi(B_a)=\beta a$. Since $1\geq \phi(B_a)=\beta a \geq 4a$, we have $B_a\in[0,\phi^{-1}(4a)]$. Hence,
$$K(a)-a = \min_{B\in[0,\phi^{-1}(4a)]} W(a,B) - a \leq  Q(a),$$
where
\begin{align*}
Q(a) := W(a,B_a)-a&= \frac{(\beta-\sqrt{\beta^2-4\beta})^2}{4\beta}a+\frac{\phi(2B_a)}{\beta}-a.
\end{align*}
Moreover, we have $B_a'=\beta/\phi'(B_a)=\beta/\log B_a$, whence
\begin{align*}
Q'(a)&= \frac{(\beta-\sqrt{\beta^2-4\beta})^2}{4\beta}+2B'_a\frac{\log(2B_a)}{\beta}-1%= \frac{(\beta-\sqrt{\beta^2-4\beta})^2}{4\beta}+ 2\frac{\log(2B_a)}{ \log(B_a)}-1\\
= \frac{(\beta-\sqrt{\beta^2-4\beta})^2}{4\beta}+ 2\frac{\log(2)}{ \log(B_a)}+1.
\end{align*}
Letting 
$$\eta:=\exp\left(\frac{-8\beta\log 2}{(\beta-\sqrt{\beta^2-4\beta})^2+4\beta}\right),$$ 
we have 
%$$Q'(a)>0 \quad \text{ if } B_a<\eta $$
%and 
%$$W'(a)<0  \quad \text{ if }  B_a\in(\eta,1).$$ 
\begin{align*}
Q'(a)<0  \quad &\text{ if }  \eta< B_a <1\\
Q'(a)>0 \quad &\text{ if } 0< B_a<\eta.
\end{align*} 
Equivalently, since $\phi$ is decreasing and $\phi(B_a)=\beta a$, we obtain 
\begin{align*}
Q'(a)<0\quad &\text{ if } 0<a<\tfrac1\beta\phi(\eta)=0.05236391\dots\\
Q'(a)>0 \quad &\text{ if }  \tfrac1\beta\phi(\eta)<a<\tfrac{1}{\beta}=0.1884244\dots
\end{align*} 
Now, since 
%$Q(0.11717)=-4.02\dots{}\cdot 10^{-6}$ and $Q(0.02)=-0.011\dots$,
$Q(0.11717)=-4.02\dots{}\cdot 10^{-6}$ and $Q(0.02)=-0.011\dots$ are both $<0$,
it follows that 
$$Q(a)\leq\min(Q(0.02),Q(0.11717))<0 \quad \text{ for } a\in[0.02,0.11717],$$ hence the conclusion.

\subsection{Proof of Equation \texorpdfstring{\eqref{thh.02}}{}}
We already know that the $B$ which provides the minimum in $K$ satisfies $B\leq 1/2$. Then, by Equation~\eqref{29Jan.1} we have 
$$t_{a,B}=\frac{1-\sqrt{1-4a/\phi(B)}}2=\frac{a}{\phi(B)}+O(a^2)$$
%+\frac{a^2}{\phi(B)^2}+O(a^3)$ 
as $a\to0$, uniformly in $B\in[0,1/2]$. Thus,
\begin{align*}
W(a,B)&=\phi(B)t_{a,B}^2+a\frac{\phi(2B)}{\phi(B)}=\frac{a^2}{\phi(B)}+a\frac{\phi(2B)}{\phi(B)}+O(a^3)\\
&\geq \frac{a^2}{\phi(B)}+O(a^3)\geq \frac{a^2}{\phi(1/2)}+O(a^3).
\end{align*}
Taking $B=1/2$ we have $W(a,1/2)\sim\frac{a^2}{\phi(1/2)}$, and thus we obtain the claimed asymptotic.

\end{document}